\documentclass[12pt]{amsart}
\usepackage[margin=1in]{geometry}
\usepackage[english]{babel}
\usepackage[autostyle]{csquotes}
\usepackage[utf8]{inputenc}
\usepackage{subcaption}
\usepackage{amsmath}
\usepackage[dvipsnames]{xcolor}
\usepackage{algorithm}
\usepackage{algpseudocode}
\usepackage{amssymb}
\usepackage{mathrsfs}
\usepackage{float}
\usepackage{amsfonts}
\usepackage{amsthm}
\usepackage{asymptote}
\usepackage{mathdots}
\usepackage[all]{xy}
\usepackage[labelfont=bf,labelsep=period,justification=raggedright]{caption}
\usepackage{hyperref}
\usepackage[colorinlistoftodos]{todonotes}
\usepackage{tkz-fct}
\usepackage{tikz}
\usepackage{tikz-3dplot}
\usetikzlibrary{calc}
\usetikzlibrary{knots}
\usetikzlibrary{spath3, intersections, hobby}
\usetikzlibrary{arrows, automata}
\usetikzlibrary{arrows.meta}
\usetikzlibrary{shapes.multipart}
\usetikzlibrary{lindenmayersystems}
\usepackage{pgfplots}
\usepackage{multicol}
\PassOptionsToPackage{dvipsnames,svgnames}{xcolor}
\usepackage{textcomp}
\usepackage{bbm}
\usepackage{plain}
\usepackage{array}
\newcolumntype{C}{>{$}c<{$}}
\usepackage{forest}
\usepackage{tikzsymbols}
\usetikzlibrary{3d}

\theoremstyle{plain}
\newtheorem{thm}{Theorem}[section]
\newtheorem{lemma}[thm]{Lemma}

\newtheorem{prop}[thm]{Proposition}

\theoremstyle{definition}
\newtheorem{defn}[thm]{Definition}

\theoremstyle{remark}
\newtheorem{remark}[thm]{Remark}

\numberwithin{equation}{section}

\definecolor{nodeblue}{RGB}{52,101,164}
\definecolor{edgegray}{gray}{0.35}
\definecolor{evolve}{RGB}{196,46,46}
\definecolor{periodic}{RGB}{0,130,0}

\tikzset{
  vtx/.style={circle,draw=nodeblue,very thick,minimum size=7mm,inner sep=0pt,font=\small},
  edg/.style={line width=0.9pt,color=edgegray},
  evo/.style={-{Latex[length=2.5mm]},line width=0.9pt,color=evolve},
  per/.style={-{Latex[length=2.5mm]},line width=0.9pt,color=periodic,dashed}
}

\title[Dynamic Coprime Labeling]{Dynamic Coprime Labeling: A Novel Framework for Time-Sensitive Networks}
\author{Anushka Tonapi and Dana Paquin}
\date{November 2025}

\begin{document}
\begin{abstract}
In this paper, we introduce dynamic coprime labeling (DCL), a novel extension of coprime labeling for time-sensitive networks. In particular, we explore the question of whether there exists a graph labeling scheme that maintains relative coprimality among adjacent vertices as the graph evolves over time. We modify the definition of coprime labeling for DCL to include an injective coprime labeling function $f$, a time variable $t$, and a transformation function $g$. A DCL on a finite simple graph $G$ is a sequence $(f_t)_{t\ge 0}$ of injective vertex labelings with the property that every edge remains labeled by coprime integers at every time step, and the evolution is a \emph{coprime-preserving} transformation $g:\mathbb{N}\to\mathbb{N}$ independent of $t$. We prove that a graph admits a DCL if and only if it admits a classical coprime labeling (existence equivalence). We characterize families of coprime-preserving transformations and provide proofs of the existence of DCLs for paths, wheels, cycles, and the $n$--hypercube. We also include two different types of maps for the coprime-preserving transformation, and look at an application of DCL to Carmichael's theorem. The results consolidate DCL as a rigorous framework for further algorithmic and applied investigations.
\end{abstract}

\maketitle

\section{Introduction}
\noindent
Graph labeling is an area of study within combinatorics and graph theory, having extensive applications in fields such as cryptography and network theory. In general, a \emph{graph labeling} is an assignment of integers to the vertices or edges, or both, depending on certain constraints. Typically, a graph labeling is an injective function that maps elements from $\mathbb{Z}$ to the vertices or edges of a graph. This field of study goes back to around the 1960s. In the last six decades, over 350 graph labeling techniques have been introduced~\cite{gallian}. However, most of these labelings have adhered to static conditions only; i.e., they operate in systems that do not change dynamically with time.

Wherever graphs are used today, they are usually seen in time-sensitive systems like cryptographic networks or communication systems that are subject to change with time. This motivates the question of whether one can introduce a time parameter into classical labeling schemes while preserving their number-theoretic properties.

The \emph{coprime labeling} of a graph gives adjacent vertices relatively prime integer labels, studied since the work of Entringer and surveyed comprehensively by Gallian~\cite{gallian}. Prime and coprime labelings have been established for many families of graphs and remain open for others. This paper develops a dynamic extension, \emph{dynamic coprime labeling}, which aims to preserve classical coprime labeling while allowing labels to evolve through time under a fixed map $g$.

The concept of \emph{prime labeling} was first introduced by Roger Entringer and was further discussed in a seminal paper by Tout, Dabboucy, and Howalla~\cite{tout}. The idea of \emph{neighbourhood prime labeling} was introduced by Patel and Shrimali~\cite{patel}. The \emph{neighbourhood} of a vertex can be defined as the set of all vertices adjacent to that vertex, excluding itself. If the vertices of a graph $G$ can be labeled with $\{1,2,\ldots, n\}$ so that for any vertex of degree at least $2$, the labels in its neighbourhood are relatively prime, this is called a \emph{neighbourhood prime labeling}, and a graph that admits such a labeling is called a \emph{neighbourhood-prime graph}.

Neighbourhood prime labeling provides one route to incorporating local number-theoretic structure into graphs. Our focus here is different: we keep the classical edge-based coprimality condition, but allow vertex labels to change systematically over time via a coprime-preserving transformation.

\section{Preliminaries: Classical Prime and Coprime Labelings}
We use $\mathbb{N}=\{1,2,3,\dots\}$, $\gcd(\cdot,\cdot)$ for the greatest common divisor, and finite simple graphs $G=(V,E)$.

\begin{defn}[Prime labeling]\label{def:prime}
For $|V|=n$, a \emph{prime labeling} is a bijection $f:V\to\{1,2,\dots,n\}$ such that $\gcd(f(u),f(v))=1$ for all $uv\in E$~\cite{gallian}.
\end{defn}

\begin{defn}[Coprime labeling]\label{def:coprime}
A \emph{coprime labeling} assigns pairwise distinct values in $\{1,\dots,k\}$ (with $k\ge n$) to $V$ so that $\gcd(f(u),f(v))=1$ for all $uv\in E$~\cite{berliner}.
\end{defn}

\begin{center}
\begin{tikzpicture}
  \node[vtx,label=below:{\scriptsize 1}] (a) at (0,0) {};
  \node[vtx,label=below:{\scriptsize 2}] (b) at (2.2,0) {};
  \node[vtx,label=below:{\scriptsize 3}] (c) at (4.4,0) {};
  \node[vtx,label=below:{\scriptsize 4}] (d) at (6.6,0) {};
  \draw[edg] (a)--(b)--(c)--(d);
\end{tikzpicture}

\vspace{2mm}
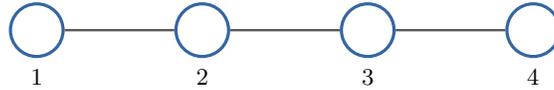
\small Figure~1. A prime-labeled path $P_4$ (labels $1,2,3,4$ satisfy $\gcd(i,i+1)=1$).
\end{center}

\section{Definitions and Notation}
We collect the definitions used throughout.

\begin{defn}[Coprime-preserving transformation]\label{def:cpp}
A map $g:\mathbb{N}\to\mathbb{N}$ is \emph{coprime-preserving} if
\[
\gcd(a,b)=1 \ \Longrightarrow\ \gcd\bigl(g(a),g(b)\bigr)=1
\quad\text{for all }a,b\in\mathbb{N}.
\]
\end{defn}

\begin{defn}[Dynamic coprime labeling (DCL)]\label{def:DCL}
Let $G=(V,E)$ be a finite simple graph. A \emph{dynamic coprime labeling (DCL)} of $G$ with respect to a coprime-preserving map $g:\mathbb{N}\to\mathbb{N}$ is a sequence of injective labelings
\[
f_t:V\to\mathbb{N} \quad (t\ge0)
\]
such that:
\begin{enumerate}
    \item $f_0$ is a coprime labeling of $G$ (i.e.\ $\gcd(f_0(u),f_0(v))=1$ for all $uv\in E$),
    \item $f_{t+1}=g\circ f_t$ for all $t\ge0$, and
    \item for every edge $uv\in E$ and every time $t\ge0$,
    \[
      \gcd(f_t(u),f_t(v))=1.
    \]
\end{enumerate}
We say that $G$ \emph{admits} a DCL (for a given $g$) if such a sequence exists.
\end{defn}

\begin{defn}[Periodic DCL]\label{def:periodic}
A DCL $(f_t)_{t\ge0}$ is \emph{periodic} if there exists $T\ge1$ with $f_{t+T}=f_t$ for all $t\ge0$; equivalently, the label set
\[
\mathcal{L}=\bigcup_{t\ge0} f_t(V)
\]
is finite and $g$ acts as a permutation on $\mathcal{L}$.
\end{defn}

\begin{defn}[$n$--hypercube graph]
The $n$--hypercube $Q_n$ is defined on the vertex set $V(Q_n)=\{0,1\}^n$ where two vertices are adjacent if and only if they differ in exactly one coordinate. It is bipartite with parts
\[
V_0=\{x\in\{0,1\}^n: \sum_i x_i \text{ even}\}, \quad
V_1=\{x\in\{0,1\}^n: \sum_i x_i \text{ odd}\}.
\]
\end{defn}

\begin{defn}[Bounded/Unbounded Labelings]\label{def:bounded-dichotomy}
Let $(f_t)_{t\ge 0}$ be a dynamic coprime labeling (DCL) of a finite graph $G=(V,E)$ with respect to $g$. Define the \emph{label--evolution set}
\[
\mathcal{L} \;:=\; \bigcup_{t\ge 0} f_t(V)
\;=\; \{\, x\in\mathbb{N} \;:\; \exists\, t\ge 0,\ \exists\, v\in V \text{ with } x=f_t(v)\,\}.
\]
We say the DCL is \emph{bounded} if $\mathcal{L}$ is finite, and \emph{unbounded} otherwise. We present two schemes:
\begin{enumerate}
  \item \textbf{Periodic/bounded case.} If there exists a finite set $S\subset\mathbb{N}$ with $f_0(V)\subseteq S$ and $g(S)=S$ such that $g$ is a permutation on $S$, then $f_{t+T}=f_t$ for some $T\ge 1$ (hence the DCL is periodic and $\mathcal{L}\subseteq S$ is finite).
  \item \textbf{Unbounded case.} If $g$ is injective and does \emph{not} act as a permutation on any finite set containing $f_0(V)$, then the sequence of label sets $\{f_t(V)\}_{t\ge 0}$ introduces infinitely many distinct values; equivalently, $\mathcal{L}$ is infinite (labels grow without bound).
\end{enumerate}
\end{defn}

\begin{remark}
We outline a framework for establishing DCL results.
\begin{enumerate}
  \item \textbf{Verify that $g$ is coprime-preserving.} Examples include $g(x)=x^k$ for integers $k\ge1$ and the map $g(x)=p_x$ sending $x$ to the $x$-th prime. In contrast, additive shifts $x\mapsto x+c$ are \emph{not} coprime-preserving in general.
  \item \textbf{Induction over time.} If $f_t$ is coprime on edges and $g$ is coprime-preserving, then $f_{t+1}=g\circ f_t$ is coprime on edges by induction. If $g$ is injective, injectivity of the labeling is also preserved.
  \item \textbf{Periodicity versus growth.} If $g$ acts as a permutation on a finite label set $S$ containing $f_0(V)$, then $(f_t)$ is periodic. If $g$ is injective and unbounded on $f_0(V)$, labels grow without bound.
\end{enumerate}
\end{remark}

\section{Main Results}

\begin{thm}[Existence Equivalence]\label{thm:existence-equivalence}
A finite graph $G$ admits a dynamic coprime labeling (with respect to some coprime-preserving map $g$) if and only if it admits a classical coprime labeling.
\end{thm}

\begin{proof}
($\Rightarrow$) Suppose $(f_t)_{t\ge0}$ is a DCL of $G$ with respect to a coprime-preserving map $g$. By Definition~\ref{def:DCL}, the initial labeling $f_0$ is injective and satisfies $\gcd(f_0(u),f_0(v))=1$ for every edge $uv\in E(G)$. Hence $f_0$ is a classical coprime labeling of $G$.

($\Leftarrow$) Conversely, assume $G=(V,E)$ admits a coprime labeling $f_0:V\to\mathbb{N}$. Define $g(x)=x^2$. By Lemma~\ref{lem:power-coprime} below, $g$ is both injective and coprime-preserving. Define
\[
f_t(v)=g^{(t)}(f_0(v)) = \bigl(f_0(v)\bigr)^{2^t}
\quad\text{for }t\ge0.
\]
If $uv\in E$, then
\[
\gcd(f_t(u),f_t(v))
= \gcd\bigl(f_0(u)^{2^t},f_0(v)^{2^t}\bigr)
= \bigl(\gcd(f_0(u),f_0(v))\bigr)^{2^t}
= 1,
\]
since $f_0$ is coprime on edges. Thus $\gcd(f_t(u),f_t(v))=1$ for all edges $uv$ and all $t\ge0$. Injectivity of each $f_t$ follows from injectivity of $f_0$ and $g$. Hence $(f_t)_{t\ge0}$ is a DCL of $G$ with respect to $g(x)=x^2$.
\end{proof}

\begin{lemma}[Coprimality under time evolution]\label{lem:coprime-persistence}
Let $G=(V,E)$ be any finite graph and let $g:\mathbb N\to\mathbb N$ be a coprime-preserving map. Suppose the initial labeling $f_0:V\to\mathbb N$ satisfies
\[
\gcd(f_0(u),f_0(v))=1\quad\text{for all }uv\in E.
\]
Define $f_{t+1}=g\circ f_t$ for $t\ge0$. Then for all $t\ge0$ and all $uv\in E$,
\[
\gcd(f_t(u),f_t(v))=1.
\]
\end{lemma}

\begin{proof}
We proceed by induction on $t$.

\emph{Base case $t=0$.} This holds by hypothesis on $f_0$.

\emph{Inductive step.} Assume $\gcd(f_t(u),f_t(v))=1$ for all edges $uv$ at some time $t$. For any edge $uv$,
\[
f_{t+1}(u)=g(f_t(u)), \quad f_{t+1}(v)=g(f_t(v)).
\]
Since $g$ is coprime-preserving and $\gcd(f_t(u),f_t(v))=1$, we have
\[
\gcd(f_{t+1}(u),f_{t+1}(v))
= \gcd\bigl(g(f_t(u)),g(f_t(v))\bigr)=1.
\]
By induction, coprimality holds for all $t\ge0$.
\end{proof}

\begin{thm}[Paths]\label{thm:paths}
Every path \(P_n\) admits a DCL under \(g(x)=x^2\).
\end{thm}

\begin{proof}
Let \(P_n=(V,E)\) where \(V=\{v_1,\dots,v_n\}\) and \(E=\{\{v_i,v_{i+1}\}\mid 1\le i<n\}\). Define
\[
f_0(v_i)=i,\quad 1\le i\le n.
\]
For each edge $\{v_i,v_{i+1}\}$ we have $\gcd(i,i+1)=1$, so $f_0$ is a coprime labeling. Let $g(x)=x^2$ and define
\[
f_t(v_i)=g^{(t)}(f_0(v_i)) = i^{2^t}.
\]
Then, for each edge,
\[
\gcd(f_t(v_i),f_t(v_{i+1}))
= \gcd\bigl(i^{2^t},(i+1)^{2^t}\bigr)
= \bigl(\gcd(i,i+1)\bigr)^{2^t}
= 1.
\]
Injectivity of $f_t$ follows from injectivity of $f_0$ and $g$. Hence $P_n$ admits a DCL under $g(x)=x^2$.
\end{proof}

\begin{thm}[Odd Cycles]\label{thm:odd-cycles}
Every odd cycle \(C_n\) (with $n$ odd) admits a non-trivial DCL under \(g(x)=x^2\).
\end{thm}

\begin{proof}
Let \(C_n=(V,E)\) with $V=\{v_1,\dots,v_n\}$ and edges joining $v_i$ to $v_{i+1}$ for $1\le i<n$, and $v_n$ to $v_1$. Define $f_0(v_i)=i$ for $1\le i\le n$. For each edge $\{v_i,v_{i+1}\}$ we have $\gcd(i,i+1)=1$, and for the closing edge $\{v_n,v_1\}$,
\[
\gcd(f_0(v_n),f_0(v_1))=\gcd(n,1)=1
\]
since $n$ is odd. Thus $f_0$ is a coprime labeling of $C_n$. As in Theorem~\ref{thm:paths}, define $f_t(v_i)=i^{2^t}$. Coprimality is preserved as before, and injectivity holds. This yields a non-trivial DCL on $C_n$ under $g(x)=x^2$.
\end{proof}

\begin{thm}[Wheels]\label{thm:wheels}
Every wheel graph \(W_n\) for \(n\ge3\) admits a DCL under \(g(x)=x^2\).
\end{thm}

\begin{proof}
Let $W_n$ consist of a central vertex $v_0$ and a cycle $C_n$ formed by vertices $v_1,\dots,v_n$. Let $p_1,p_2,\dots$ denote the sequence of primes in increasing order. Define
\[
f_0(v_0)=1,\qquad f_0(v_i)=p_i\quad (1\le i\le n).
\]
Each spoke edge $\{v_0,v_i\}$ has labels $1$ and $p_i$, which are coprime. Each rim edge $\{v_i,v_{i+1}\}$ (with indices taken modulo $n$) connects two distinct primes, hence $\gcd(p_i,p_{i+1})=1$. Thus $f_0$ is a coprime labeling of $W_n$.

Let $g(x)=x^2$ and define $f_t=g^{(t)}\circ f_0$ as before. By Lemma~\ref{lem:power-coprime}, $g$ is injective and coprime-preserving, so Lemma~\ref{lem:coprime-persistence} ensures that $\gcd(f_t(u),f_t(v))=1$ for all edges $uv$ and all $t\ge0$. Injectivity is preserved as well. Hence $(f_t)_{t\ge0}$ is a DCL of $W_n$ under $g(x)=x^2$.
\end{proof}

\begin{thm}[Hypercubes]\label{thm:hypercube}
Every $n$--dimensional hypercube $Q_n$ admits a dynamic coprime labeling (DCL) under any coprime-preserving transformation $g:\mathbb N\to\mathbb N$ (for example, $g(x)=x^2$).
\end{thm}

\begin{proof}
First, we construct a static coprime labeling. Let $\mathcal P_0$ and $\mathcal P_1$ be two disjoint infinite sets of primes with
\[
\mathcal P_0\cap\mathcal P_1=\varnothing.
\]
Assign injectively
\[
\phi_0:V_0\to \mathcal P_0, \qquad \phi_1:V_1\to \mathcal P_1,
\]
and define the labeling $f_0:V(Q_n)\to\mathbb N$ by
\[
f_0(v)=\begin{cases}
\phi_0(v), & v\in V_0,\\[3pt]
\phi_1(v), & v\in V_1.
\end{cases}
\]
Since edges of $Q_n$ always connect one vertex from $V_0$ and one from $V_1$, we have $\gcd(f_0(u),f_0(v))=1$ for every edge $uv\in E(Q_n)$. Thus $f_0$ is a static coprime labeling of $Q_n$.

Now define $f_{t+1}=g\circ f_t$ for $t\ge0$, with $f_0$ as above. By Lemma~\ref{lem:coprime-persistence}, coprimality is preserved at every time step. If $g$ is injective (for example, $g(x)=x^2$), then each $f_t$ remains injective as well. Hence $(f_t)_{t\ge 0}$ forms a valid dynamic coprime labeling of the hypercube.
\end{proof}

\begin{remark}
Because $Q_n$ is bipartite, the same argument extends to all bipartite graphs. If $G$ is bipartite with parts $(A,B)$, label vertices in $A$ and $B$ with disjoint sets of primes, and apply any coprime-preserving map $g$. Thus every bipartite graph admits a dynamic coprime labeling.
\end{remark}

\begin{lemma}[Power maps preserve coprimality]\label{lem:power-coprime}
For any integer $k\ge 1$, the map $g:\mathbb{N}\to\mathbb{N}$ given by $g(x)=x^k$ is injective and coprime-preserving. In particular, for all $a,b\in\mathbb{N}$,
\[
\gcd(a^k,b^k)=\bigl(\gcd(a,b)\bigr)^{k},
\]
so $\gcd(a,b)=1$ implies $\gcd(a^k,b^k)=1$.
\end{lemma}

\begin{proof}
Injectivity is immediate: if $a^k=b^k$ for $a,b\in\mathbb{N}$ and $k\ge 1$, then $a=b$.

For coprime preservation, write the prime-power factorizations
\[
a=\prod p^{\alpha},\qquad b=\prod p^{\beta}\qquad(\alpha,\beta\ge 0),
\]
so that
\[
\gcd(a,b)=\prod p^{\min(\alpha,\beta)}.
\]
Then
\[
a^k=\prod p^{k\alpha},\qquad b^k=\prod p^{k\beta},\qquad
\gcd(a^k,b^k)=\prod p^{\min(k\alpha,k\beta)}=\prod p^{k\min(\alpha,\beta)}=\bigl(\gcd(a,b)\bigr)^k.
\]
In particular, if $\gcd(a,b)=1$ then $\gcd(a^k,b^k)=1$.
\end{proof}

\begin{lemma}[Generalization from $x^2$ to $x^k$]\label{cor:generalize-xk}
Let $G=(V,E)$ be any finite graph and let $f_0:V\to\mathbb{N}$ be injective with $\gcd\!\bigl(f_0(u),f_0(v)\bigr)=1$ for every edge $uv\in E$.
For any fixed $k\ge 1$, define $f_{t+1}=g\circ f_t$ with $g(x)=x^k$. Then
\[
f_t(v)=(f_0(v))^{k^t}\quad\text{and}\quad \gcd\!\bigl(f_t(u),f_t(v)\bigr)=1\ \ \text{for all }uv\in E,\ t\ge 0.
\]
Hence $g(x)=x^k$ yields a valid dynamic coprime labeling for every $k\ge 1$.
\end{lemma}

\begin{remark}
The case $k=2$ used earlier is thus a special instance of Lemma~\ref{lem:power-coprime}. 
For $k\ge 2$, the DCL is typically unbounded (labels grow as $k^t$ in the exponents); for $k=1$ it is static. 
Working modulo a fixed integer $m$, the map $x\mapsto x^k\!\!\pmod m$ gives periodic (bounded) labelings, and the resulting period connects to multiplicative orders modulo $m$.
\end{remark}

\subsection*{Dynamic coprime labeling on $Q_3$ under $g(x)=x^2$}
Take the bipartition
\[
V_0=\{000,011,101,110\},\qquad V_1=\{001,010,100,111\},
\]
and assign
\[
\begin{aligned}
f_0(000)&=2,&\quad f_0(011)&=3,&\quad f_0(101)&=5,&\quad f_0(110)&=7,\\
f_0(001)&=11,& f_0(010)&=13,& f_0(100)&=17,& f_0(111)&=19.
\end{aligned}
\]
Let $g(x)=x^2$ and define $f_{t+1}=g\circ f_t$.

\begin{center}
\begin{tikzpicture}[scale=1.0]
  \def\dx{2.6}   
  \def\dy{1.6}   
  \def\shift{8.2}

  \tikzset{
    vtx/.style={circle, draw=blue!60!black, very thick, minimum size=7mm, inner sep=0pt},
    edg/.style={line width=0.9pt, draw=black!60},
    lbl/.style={font=\scriptsize, label distance=1.2mm}
  }

  \node at (0,2.7) {\small $t=0$};

  \node[vtx,label={[lbl]below:2}] (A000) at (-1,  1) {};
  \node[vtx,label={[lbl]below:13}] (A010) at ( 1,  1) {};
  \node[vtx,label={[lbl]below:11}] (A001) at (-1, -1) {};
  \node[vtx,label={[lbl]below:3}]  (A011) at ( 1, -1) {};

  \node[vtx,label={[lbl]below:17}] (A100) at (-1+\dx,  1+\dy) {};
  \node[vtx,label={[lbl]below:7}]  (A110) at ( 1+\dx,  1+\dy) {};
  \node[vtx,label={[lbl]below:19}] (A111) at ( 1+\dx, -1+\dy) {};
  \node[vtx,label={[lbl]below:5}]  (A101) at (-1+\dx, -1+\dy) {};

  \draw[edg] (A000)--(A010)--(A011)--(A001)--(A000);
  \draw[edg] (A100)--(A110)--(A111)--(A101)--(A100);
  \draw[edg] (A000)--(A100);
  \draw[edg] (A010)--(A110);
  \draw[edg] (A011)--(A111);
  \draw[edg] (A001)--(A101);

  \node at (\shift,2.7) {\small $t=1$};

  \node[vtx,label={[lbl]below:$2^2$}]   (B000) at (\shift-1,  1) {};
  \node[vtx,label={[lbl]below:$13^2$}]  (B010) at (\shift+1,  1) {};
  \node[vtx,label={[lbl]below:$11^2$}]  (B001) at (\shift-1, -1) {};
  \node[vtx,label={[lbl]below:$3^2$}]   (B011) at (\shift+1, -1) {};

  \node[vtx,label={[lbl]below:$17^2$}]  (B100) at (\shift-1+\dx,  1+\dy) {};
  \node[vtx,label={[lbl]below:$7^2$}]   (B110) at (\shift+1+\dx,  1+\dy) {};
  \node[vtx,label={[lbl]below:$19^2$}]  (B111) at (\shift+1+\dx, -1+\dy) {};
  \node[vtx,label={[lbl]below:$5^2$}]   (B101) at (\shift-1+\dx, -1+\dy) {};

  \draw[edg] (B000)--(B010)--(B011)--(B001)--(B000);
  \draw[edg] (B100)--(B110)--(B111)--(B101)--(B100);
  \draw[edg] (B000)--(B100);
  \draw[edg] (B010)--(B110);
  \draw[edg] (B011)--(B111);
  \draw[edg] (B001)--(B101);
\end{tikzpicture}

\vspace{2mm}
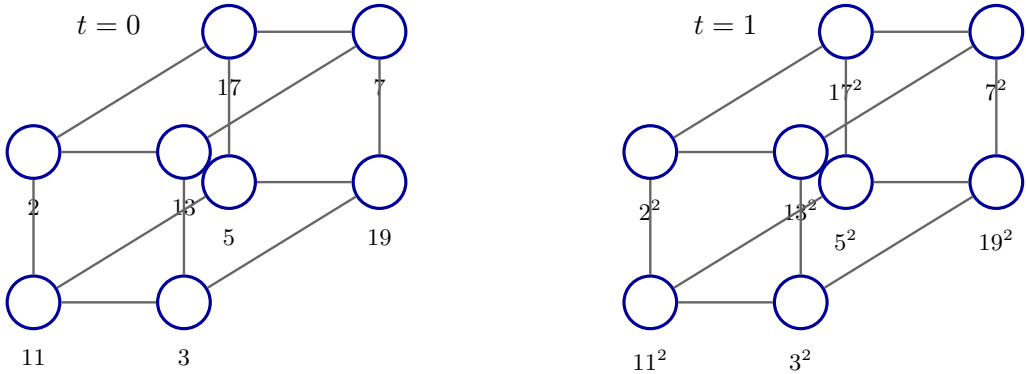
\small Figure~5. A dynamic coprime labeling on $Q_3$ under $g(x)=x^2$: 
at $t=0$ each parity class is labeled with disjoint primes; 
at $t=1$ every label is squared. 
\end{center}

\subsection*{Evolution on $P_4$ under $g(x)=x^2$}
\begin{center}
\begin{tikzpicture}[x=1.1cm,y=1cm]
\node at (0,1.2) {\small $t=0$};
\node[vtx,label=below:{\scriptsize 1}] (a0) at (0,0) {};
\node[vtx,label=below:{\scriptsize 2}] (b0) at (1.6,0) {};
\node[vtx,label=below:{\scriptsize 3}] (c0) at (3.2,0) {};
\node[vtx,label=below:{\scriptsize 4}] (d0) at (4.8,0) {};
\draw[edg] (a0)--(b0)--(c0)--(d0);
\node at (7.0,1.2) {\small $t=1$};
\node[vtx,label=below:{\scriptsize $1^2$}] (a1) at (7,0) {};
\node[vtx,label=below:{\scriptsize $2^2$}] (b1) at (8.6,0) {};
\node[vtx,label=below:{\scriptsize $3^2$}] (c1) at (10.2,0) {};
\node[vtx,label=below:{\scriptsize $4^2$}] (d1) at (11.8,0) {};
\draw[edg] (a1)--(b1)--(c1)--(d1);
\draw[evo] (a0) to[bend left=20] (a1);
\draw[evo] (b0) to[bend left=20] (b1);
\draw[evo] (c0) to[bend left=20] (c1);
\draw[evo] (d0) to[bend left=20] (d1);
\end{tikzpicture}

\vspace{2mm}
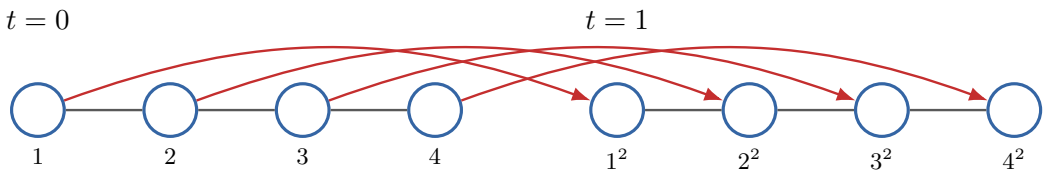
\small Figure~2. Evolution of a DCL on $P_4$ via $g(x)=x^2$.
\end{center}

\subsection*{Periodic DCL on $C_3$}
\begin{center}
\begin{tikzpicture}[scale=1.0]
  \node[vtx,label=above:{\scriptsize $a$}] (x) at (0,1.4) {};
  \node[vtx,label=right:{\scriptsize $b$}] (y) at (1.4,-0.8) {};
  \node[vtx,label=left:{\scriptsize $c$}] (z) at (-1.4,-0.8) {};
  \draw[edg] (x)--(y)--(z)--(x);
  \draw[per] (x) to[bend left=25] (y);
  \draw[per] (y) to[bend left=25] (z);
  \draw[per] (z) to[bend left=25] (x);
\end{tikzpicture}

\vspace{2mm}
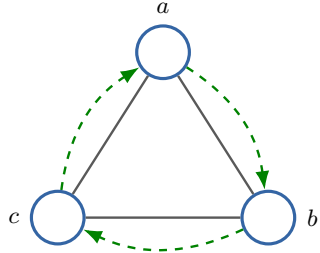
\small Figure~3. A periodic DCL on $C_3$: a permutation of a finite label set preserving coprimality.
\end{center}

\subsection*{Evolution on $W_5$ under $g(x)=x^2$}
\begin{center}
\begin{tikzpicture}[scale=1.0]
  \def\R{1.8}      
  \def\shift{7.2}  

  \node at (0,2.25) {\small $t=0$};

  \node[vtx,label=below:{\scriptsize 1}] (H0) at (0,0) {};
  \foreach \i/\ang/\lab in {1/90/2, 2/18/3, 3/-54/5, 4/-126/7, 5/162/11} {
    \node[vtx,label=below:{\scriptsize \lab}] (R0\i) at ({\R*cos(\ang)},{\R*sin(\ang)}) {};
  }

  \foreach \i/\j in {1/2,2/3,3/4,4/5,5/1} { \draw[edg] (R0\i) -- (R0\j); }
  \foreach \i in {1,2,3,4,5} { \draw[edg] (H0) -- (R0\i); }

  \node at (\shift,2.25) {\small $t=1$};

  \node[vtx,label=below:{\scriptsize $1^2$}] (H1) at (\shift,0) {};
  \foreach \i/\ang/\lab in {1/90/2^2, 2/18/3^2, 3/-54/5^2, 4/-126/7^2, 5/162/11^2} {
    \node[vtx,label=below:{\scriptsize $\lab$}] (R1\i) at ({\shift+\R*cos(\ang)},{\R*sin(\ang)}) {};
  }

  \foreach \i/\j in {1/2,2/3,3/4,4/5,5/1} { \draw[edg] (R1\i) -- (R1\j); }
  \foreach \i in {1,2,3,4,5} { \draw[edg] (H1) -- (R1\i); }

  \draw[evo, bend left=10] (H0) to (H1);
  \foreach \i in {1,2,3,4,5} { \draw[evo, bend left=18] (R0\i) to (R1\i); }
\end{tikzpicture}

\vspace{2mm}
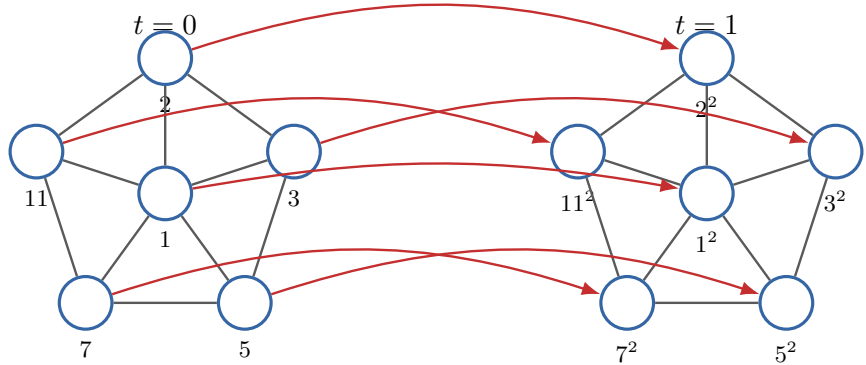
\small Figure~4. A wheel graph $W_5$ with a dynamic coprime labeling:
hub $=1$ and rim labels given by distinct primes at $t=0$; under $g(x)=x^2$ each label
evolves to its square at $t=1$.
\end{center}

\section*{Application to the Carmichael Function}

In this section we work modulo a fixed integer $n$ and assume all vertex labels are coprime to $n$, so that they lie in the multiplicative group $(\mathbb{Z}/n\mathbb{Z})^\times$.

\begin{defn}[Vertex order and graph period]
Let $G=(V,E)$ and $f_0:V\to\mathbb{N}$ be a labeling with $\gcd(f_0(v),n)=1$ for all $v\in V$. For a vertex $v$, define its \emph{vertex order} modulo $n$ by
\[
T_v = \operatorname{ord}_n(f_0(v)) = \min\{\,T>0:\ f_0(v)^T \equiv 1\pmod n\,\}.
\]
The \emph{graph period} is the least common multiple
\[
\lambda_G(n) = \mathrm{lcm}_{v\in V(G)} T_v.
\]
\end{defn}

\begin{defn}[Carmichael function]
The \emph{Carmichael function} $\lambda(n)$ is the least positive integer $m$ such that
\[
a^{m} \equiv 1 \pmod n
\quad\text{for all }a\text{ with }\gcd(a,n)=1.
\]
Equivalently, $\lambda(n)$ is the exponent of the multiplicative group $(\mathbb{Z}/n\mathbb{Z})^\times$.
\end{defn}

\begin{prop}[Graph--Carmichael Function]\label{prop:graph-carmichael}
Let $G=(V,E)$ admit a DCL modulo $n$ with initial labeling $f_0:V\to(\mathbb{Z}/n\mathbb{Z})^\times$. Then:
\[
\lambda_G(n) \mid \lambda(n).
\]
If additionally the set $\{f_0(v):v\in V(G)\}$ generates $(\mathbb{Z}/n\mathbb{Z})^\times$ as a group, then
\[
\lambda_G(n)=\lambda(n).
\]
\end{prop}

\begin{proof}
Each $f_0(v)$ lies in $(\mathbb{Z}/n\mathbb{Z})^\times$ and has order $T_v$ dividing $\lambda(n)$ by definition of $\lambda(n)$ as the exponent of the group. Therefore
\[
\lambda_G(n)=\mathrm{lcm}_{v\in V(G)} T_v
\]
is a least common multiple of divisors of $\lambda(n)$, so $\lambda_G(n)\mid \lambda(n)$.

If the labels $f_0(v)$ generate the entire group $(\mathbb{Z}/n\mathbb{Z})^\times$, then some element among them must have order equal to the exponent $\lambda(n)$, and all other orders divide $\lambda(n)$. Hence their least common multiple is exactly $\lambda(n)$, giving $\lambda_G(n)=\lambda(n)$.
\end{proof}

\subsection*{Graph-theoretic view of Carmichael numbers}

Recall Korselt's criterion: a composite integer $n$ is a \emph{Carmichael number} if and only if
\begin{enumerate}
\item $n$ is square-free, and
\item for every prime $p\mid n$, $p-1\mid n-1$.
\end{enumerate}
Equivalently, $n$ is Carmichael if $a^{n-1}\equiv1\pmod n$ for all $a$ with $\gcd(a,n)=1$.

In the language of the Carmichael function, $n$ is Carmichael if $\lambda(n)\mid n-1$ and $n$ is composite and square-free. If we realize a generating DCL on a graph $G$ modulo $n$, then by Proposition~\ref{prop:graph-carmichael} we have $\lambda_G(n)=\lambda(n)$, so the period of the DCL modulo $n$ divides $n-1$. Thus, for Carmichael $n$, the evolution of labels in a generating DCL exhibits a global period that divides $n-1$; this period is exactly the exponent of $(\mathbb{Z}/n\mathbb{Z})^\times$.

\section{Example: $n=561$}

Let $n=561=3\cdot11\cdot17$. For each prime factor $p\mid 561$, the group $(\mathbb{Z}/p\mathbb{Z})^\times$ is cyclic of order $p-1$, so
\[
|(\mathbb{Z}/3\mathbb{Z})^\times|=2,\quad
|(\mathbb{Z}/11\mathbb{Z})^\times|=10,\quad
|(\mathbb{Z}/17\mathbb{Z})^\times|=16.
\]
The multiplicative group $(\mathbb{Z}/561\mathbb{Z})^\times$ is isomorphic to
\[
(\mathbb{Z}/3\mathbb{Z})^\times \times (\mathbb{Z}/11\mathbb{Z})^\times \times (\mathbb{Z}/17\mathbb{Z})^\times,
\]
so its exponent is
\[
\lambda(561)=\mathrm{lcm}(2,10,16)=80.
\]
Since $80$ divides $560=n-1$ and $561$ is square-free, Korselt's criterion implies that $561$ is a Carmichael number.

Now take a graph $G$ and a DCL modulo $561$ whose initial labels $f_0(v)$ generate $(\mathbb{Z}/561\mathbb{Z})^\times$. By Proposition~\ref{prop:graph-carmichael}, the DCL has graph period
\[
\lambda_G(561)=\lambda(561)=80.
\]
Thus
\[
f_{t+80}(v)\equiv f_t(v)\pmod{561}\quad\text{for all }v\in V(G),\ t\ge0.
\]
This provides a combinatorial visualization of the periodicity underlying the Carmichael property of $561$: every vertex label returns to its initial value after $80$ time steps modulo $561$, and this global period divides $n-1=560$.

\section*{An Alternate DCL Map}

The power maps $g(x)=x^k$ are natural examples of coprime-preserving transformations. One can also construct other affine maps that remain coprime-preserving along edges, under arithmetic conditions on the base labeling.

\begin{thm}[Alternate Dynamic Coprime Labeling Map]\label{thm:alternate-map}
Let $G = (V,E)$ be a finite simple graph and let $f_0 : V \to \mathbb{N}$ be an injective coprime labeling. Fix a prime $p$ and define the affine map $g(x) = p x + 1$. For each $t \ge 0$, define
\[
f_t(v) = g^t(f_0(v)) = p^t f_0(v) + c_t, \quad \text{where} \quad c_t = \frac{p^t - 1}{p - 1}.
\]
Suppose that for every edge $uv \in E$ and every prime $q$ dividing $|f_0(v) - f_0(u)|$, the following condition holds:
\[
(p - 1)f_0(u) + 1 \not\equiv p^k \pmod{q}\quad \text{for all }k\in\mathbb{Z},
\]
i.e.\ $(p - 1)f_0(u) + 1 \notin \langle p \rangle \subset (\mathbb{Z}/q\mathbb{Z})^\times$. Then the sequence $\{f_t\}_{t \ge 0}$ defines a dynamic coprime labeling of $G$; that is, for all $t \ge 0$,
\[
\gcd(f_t(u), f_t(v)) = 1 \quad \text{for all } uv \in E,
\]
and each $f_t$ is injective on $V$.
\end{thm}

\begin{proof}
For fixed $t\ge0$, the map $x \mapsto p^t x + c_t$ is strictly increasing on $\mathbb{N}$, hence injective. Since $f_0$ is injective, $f_t = g^t \circ f_0$ is injective for every $t$.

Now consider an edge $uv \in E$, and set $a = f_0(u)$, $b = f_0(v)$, and $d = b - a$. Without loss of generality, assume $b > a$ so that $d > 0$. Then
\[
\gcd(f_t(u), f_t(v)) = \gcd(p^t a + c_t,\, p^t b + c_t).
\]
Using the Euclidean identity $\gcd(x, y) = \gcd(x, y - x)$, we have
\[
\gcd(p^t a + c_t,\, p^t b + c_t) = \gcd(p^t a + c_t,\, p^t d).
\]
Since
\[
\gcd(p^t a + c_t,\, p^t) = 1,
\]
(because $c_t=\frac{p^t-1}{p-1}$ is not divisible by $p$), it follows that
\[
\gcd(p^t a + c_t,\, p^t d) = \gcd(p^t a + c_t,\, d).
\]

Suppose, for contradiction, that there is a prime $q$ dividing $\gcd(f_t(u),f_t(v))$. Then $q$ divides $d$ and divides $p^t a + c_t$. Substituting $c_t = \frac{p^t - 1}{p - 1}$, the congruence $p^t a + c_t \equiv 0 \pmod{q}$ becomes
\[
p^t a + \frac{p^t - 1}{p - 1} \equiv 0 \pmod{q}.
\]
Multiplying both sides by $(p - 1)$ (which is invertible modulo $q$ since $q\nmid p-1$ for primes $q\neq p$), we get
\[
p^t\bigl((p - 1)a + 1\bigr) \equiv 1 \pmod{q},
\]
i.e.
\[
(p - 1)a + 1 \equiv p^{-t} \pmod{q}.
\]
Thus $(p-1)a+1$ lies in the multiplicative subgroup $\langle p\rangle\subset(\mathbb{Z}/q\mathbb{Z})^\times$ generated by $p$ modulo $q$, contradicting the stated hypothesis. Hence no such prime $q$ can divide both $f_t(u)$ and $f_t(v)$ for any $t\ge0$.

Therefore, $\gcd(f_t(u),f_t(v)) = 1$ for all edges $uv \in E$ and all $t \ge 0$, and each $f_t$ is injective on $V$. This shows that $\{f_t\}_{t \ge 0}$ defines a dynamic coprime labeling of $G$.
\end{proof}

\begin{remark}[Problem Complexity]\label{rem:complexity}
One can view the study of DCLs also through a computational lens. Given a triple $(G,f_0,g)$, there are two natural algorithmic questions:
\begin{enumerate}
  \item \emph{Existence.} Given a graph $G=(V,E)$ and a transformation $g:\mathbb{N}\to\mathbb{N}$, decide whether there exists an initial coprime labeling $f_0$ such that the induced sequence $(f_t)_{t\ge0}$ defined by $f_{t+1}=g\circ f_t$ satisfies the coprime-adjacency constraint
  \[
  \gcd(f_t(u),f_t(v))=1\quad\text{for all }uv\in E,\ t\ge0.
  \]
  \item \emph{Verification.} Given explicit data $(G,f_0,g)$, determine whether the resulting sequence indeed preserves the coprime-adjacency constraint for all edges and all times, or finds a time/edge where it fails.
\end{enumerate}
The existence problem is naturally a constraint-satisfaction problem (CSP), and its precise complexity (e.g.\ NP-hardness for certain classes of $g$ and $G$) is an interesting direction for future research.
\end{remark}

\section*{Acknowledgments}
I wish to extend my gratitude to Prof. Dana Paquin of Stanford University for supporting me throughout the process, providing valuable feedback, and helping me upload it to the arXiv as a preprint. I would also like to thank my family for their constant encouragement and support.

\end{document}